\documentclass[11pt]{amsart}
\usepackage[utf8]{inputenc}
\usepackage{amsthm}
\usepackage{amsmath}
\usepackage{amssymb}
\usepackage{mathrsfs}
\usepackage{hyperref}
\usepackage{amsrefs}
\usepackage{xcolor}
\newtheorem{theorem}{Theorem}[section]

\newtheorem{lemma}[theorem]{Lemma}

\newtheorem{definition}[theorem]{Definition}

\newtheorem{remark}[theorem]{Remark}

\newtheorem{corollary}[theorem]{Corollary}

\numberwithin{equation}{section}

\newcommand{\bb}[1]{\mathbb{#1}}
\newcommand{\cl}[1]{\mathcal{#1}}
\begin{document}
\title[Perturbations of Truncated Toeplitz operators]{Perturbations of Truncated Toeplitz operators}

\author{Arup Chattopadhyay*}
\address{*Department Of Mathematics\\
         Indian Institute of Technology Guwahati\\
         Guwahati-781039, India.
    }
\email{arupchatt@iitg.ac.in}
\author{Arshad Khan**}
\address{**Department Of Mathematics\\
         Indian Institute of Technology Guwahati\\
         Guwahati-781039, India.
    }
\email{arshad.khan@rnd.iitg.ac.in}



\begin{abstract}
We investigate the structure of kernels of finite-rank perturbations of truncated Toeplitz operators. By realizing a perturbed truncated Toeplitz operator on a model space as a suitable matrix-valued truncated Toeplitz operator on an enlarged model space, we identify its kernel with a coordinate projection of the kernel of an associated block Toeplitz operator. This realization enables us to transfer the nearly $S^*$-invariant structure of Toeplitz kernels to the perturbed truncated setting. More precisely, for a perturbation of rank at most $n$, we show that the kernel of the perturbed truncated Toeplitz operator is nearly $S^*$-invariant with defect at most $2n+1$. We then extend the method to truncated Toeplitz operators on multiband spaces. In this setting, a perturbation of rank at most $n$ again preserves the finite-defect nearly $S^*$-invariant structure of the kernel, with the explicit upper bound $2n+2$ for the defect.

\end{abstract}

\maketitle

\section{Introduction} \label{intro}

Truncated Toeplitz operators (TTOs) form an important class of operators associated with model spaces and provide a fruitful setting for interactions between operator theory and function theory. Let $\theta$ be an inner function on the open unit disc centred at the origin $\mathbb D$ and let $\cl K_\theta=H^2(\bb D)\ominus \theta H^2(\bb D)$ be the corresponding model space. For a symbol $g\in L^\infty(\mathbb T)$, the TTO $A^{\theta}_g$ on $\cl K_\theta$ is defined by
$$
A_g^\theta f=P_\theta(gf), \qquad f\in \cl K_\theta,
$$
where $P_\theta$ denotes the orthogonal projection from $L^2(\mathbb T)$ onto $\cl K_\theta$.

The systematic study of TTOs was initiated by Sarason \cite{Sar02} in 2007. Sarason investigated compressions of Toeplitz operators to model spaces and established fundamental structural and algebraic properties of TTOs. The subsequent development of the subject was surveyed by Garcia and Ross \cite{GraRos} in 2013. Since then, various aspects of TTOs have been investigated, including their symbols, algebraic properties, finite-rank structure, kernels, and vector-valued extensions.

For the present work, the kernels of perturbed TTOs are of particular interest. A fundamental connection in this direction is provided by nearly invariant subspaces. Let $S^*$ denote the backward shift operator on the Hardy space $H^2(\bb D)$. A subspace $\cl M$ of $H^2(\bb D)$ is called nearly $S^*$-invariant if $$f\in \cl M \text{ with }f(0)=0
\quad\Longrightarrow\quad
S^*f\in \cl M.$$ More generally, one considers nearly $S^*$-invariant subspaces with finite defect, in which the failure of exact near invariance is confined to a finite-dimensional defect subspace. The theory of nearly invariant subspaces originates in the work of Hitt and Sarason. In particular, Hitt \cite{Hit} and Sarason \cite{Sar01} established fundamental structural results closely connected with the study of kernels of Toeplitz operators. The notion of nearly invariant subspaces has also been extended beyond classical Hardy spaces to operators acting on general Hilbert spaces, where it has been further investigated in various operator-theoretic settings. For further developments in this direction, one can refer to
\cite{SupAru2, KLD, Era, LP}.

The vector-valued theory was developed by Chalendar, Chevrot and Partington \cite{CCP}, who characterized nearly invariant subspaces of vector-valued Hardy spaces and thereby extended the scalar theory. This perspective is particularly useful in the present setting because matrix-valued TTOs naturally act on vector-valued model spaces. There is also a direct connection between nearly invariant subspaces and TTOs. Hartmann and Ross \cite{HarRos} studied TTOs in the setting of nearly invariant subspaces and investigated the associated boundary behavior. More recently, O'Loughlin \cite{ROL1} developed a theory of vector-valued nearly invariant subspaces with finite defect. This theory provides a unified framework for the kernels of Toeplitz operators, TTOs, TTOs on multiband spaces, and dual TTOs. Thus, nearly $S^*$-invariant subspaces with finite defect provide a natural framework for studying kernels of a broad class of Toeplitz-type operators.

A second motivation for the present work comes from finite-rank perturbations. In the classical Toeplitz setting, Liang and Partington \cite{LP3} established that kernels of finite-rank perturbations of Toeplitz operators are nearly $S^*$-invariant with finite defect. More precisely, they showed that a finite-rank perturbation does not destroy the underlying nearly invariant structure of the kernel, but instead introduces a finite-dimensional defect. Further, this investigation was extended to vector-valued Hardy spaces by Chattopadhyay, Das and Pradhan \cite{CDP2}. Further developments in this direction include the work of  Chattopadhyay and Jana \cite{SupAru1}, who studied kernels of finite-rank perturbations of Hankel operators. One can also refer to the work of  Gu \cite{Gu}. This naturally raises the question of whether an analogous phenomenon persists for TTOs.

The question is nontrivial because a TTO acts on a model space rather than on the full Hardy space. Consequently, the standard arguments for kernels of Toeplitz operators do not transfer directly to the truncated setting. A natural way to overcome this difficulty is to exploit the relationship between matrix-valued TTOs and Toeplitz operators with matrix-valued symbols. This viewpoint was developed systematically by O'Loughlin \cite{ROL2}. In particular, O'Loughlin showed, under suitable hypotheses, how the kernel of a matrix-valued TTO can be represented through an $S^*$-invariant subspace and constructed a Toeplitz operator that is equivalent, after extension, to the matrix-valued TTO. This provides a useful mechanism for transferring kernel questions to the better-understood theory of Toeplitz kernels.

The finite-rank aspect of the present problem is also related to the study of finite-rank TTOs themselves. In \cite{Bes}, Bessonov established a characterization of finite-rank TTOs. Subsequently, Ma and Zheng \cite{MaZhe} provided an alternative proof of Bessonov's theorem. Our problem is different: rather than characterizing TTOs that themselves have finite rank, we study how the kernel structure changes when a TTO is subjected to an arbitrary finite-rank perturbation. This distinction clarifies the problem addressed in this paper.

Motivated by these developments, we investigate kernels of finite-rank perturbations of TTOs. More precisely, we consider operators of the form
$$
\widetilde{A_n}:=A_g^\theta-\sum_{j=1}^{n}p_j\otimes q_j,
$$
where $g\in L^\infty(\mathbb T)$ and $p_j,q_j\in \cl K_\theta$. Here $\displaystyle{\sum_{j=1}^{n}p_j\otimes q_j}$ is an operator of rank at most $n$. Our first objective is to determine whether $\ker \widetilde{A_n}$ retains a nearly $S^*$-invariant structure and, in particular, whether its defect can be controlled in terms of the rank of the perturbation.

Our main approach is to encode the finite-rank perturbation by adjoining finitely many finite-dimensional coordinates. We consider the matrix-valued inner function $\Theta_n
=
\operatorname{diag}
\big(
\theta,
\underbrace{z,\ldots,z}_{n\text{ copies}}
\big)$. The perturbed scalar TTO can consequently be realized through a suitable matrix-valued TTO acting on the enlarged model space $\cl K_{\Theta_n}$. This realization connects the kernel of the perturbed TTO with the kernel of an associated block Toeplitz operator. Since Toeplitz kernels have a well-established nearly $S^*$-invariant structure, this representation provides the mechanism for transferring that structure to the original finite-rank perturbed TTO.

We next extend the construction to the multiband setting. Let $\cl M=\phi \cl K_\theta\oplus\psi \cl K_\theta$, where $\phi$ and $\psi$ are unimodular functions in $L^\infty(\mathbb T)$ such that the two summands are orthogonal. Consider the unitary map
\begin{align*}
U:\cl M\longrightarrow \cl K_\theta(\mathbb D,\mathbb C^2),\qquad
U(\phi f_1+\psi f_2)=(f_1,f_2).
\end{align*}
Under this identification, TTOs on the multiband space are represented by $2\times2$ matrix-valued TTOs on the vector-valued model space. This fits naturally into the matrix-valued framework developed by O'Loughlin \cite{ROL2}.

For a perturbation of rank at most $n$,
$$
\widetilde A^{\cl M}_n
:=
A_g^\cl M-\sum_{j=1}^{n}u_j\otimes v_j,
$$
we transfer the problem through $U$ and obtain a finite-rank perturbation of a $(n+2)\times (n+2)$ matrix-valued TTO. The appropriate enlarged inner matrix function is $\Phi_n
=
\operatorname{diag}
\big(
\theta,\theta,
\underbrace{z,\ldots,z}_{n\text{ copies}}
\big)$.
The resulting matrix-valued TTO can again be represented through a block Toeplitz operator. The associated block symbol has the form
$$
\mathfrak G_n
=
\begin{pmatrix}
\Phi^*_n&0\\
\mathcal A_n&\Phi_n
\end{pmatrix},
$$
where $\mathcal A_n$ is the $(n+2)\times(n+2)$ matrix-valued symbol encoding both the original matrix-valued TTO and the finite-rank perturbation. Consequently, the kernel of $\widetilde A^{\cl M}_n$ is obtained from the kernel of the corresponding block Toeplitz operator by a coordinate projection. In particular, we show that $\ker \widetilde A^{\cl M}_n$ is nearly $S^*$-invariant with finite defect.

The paper is organized as follows. Section \ref{Pre} recalls the basic notation, terminology, and definitions, including Hardy spaces, model spaces, TTOs, and nearly $S^*$-invariant subspaces, that will be used throughout the paper. In Section \ref{Mat}, we establish the matrix-valued realization of TTO kernels in terms of kernels of block Toeplitz operators and derive the corresponding nearly invariant subspace result. Section \ref{Tru} applies this construction to finite-rank perturbations of scalar TTOs. Section \ref{Mul} develops the multiband setting and proves the main result concerning finite-rank perturbations of TTOs on multiband spaces.

\section{Novelty and Contributions }

The novelty of this work lies in developing a systematic framework for studying the kernels of finite-rank perturbations of TTOs. Although finite-rank perturbations of Toeplitz operators have been investigated extensively, the corresponding problem for perturbed TTOs has not been addressed in this framework. The main difficulty is that a finite-rank perturbation generally destroys the truncated Toeplitz structure, while the operator continues to act on the model space. To overcome this difficulty, we introduce an enlarged matrix-valued model that incorporates the finite-rank perturbation into a matrix-valued TTO. Its kernel is then realized as a coordinate projection of the kernel of an associated block Toeplitz operator. This connection allows us to transfer the nearly invariant structure of Toeplitz kernels to the setting of finite-rank perturbations of TTOs.

For scalar TTOs, we establish an exact realization of the kernel of a finite rank perturbation through an enlarged matrix-valued TTO (Theorem \ref{R41}). Combining this realization with a general result on the nearly \(S^*\)-invariant structure of matrix-valued TTO kernels and a coordinate-projection estimate, we prove that the kernel of finite rank perturbed TTO with rank at most $n$ is nearly \(S^*\)-invariant with defect at most \(2n+1\) (Theorem \ref{R43}). Furthermore, the structural results developed in this setting provide a decomposition of these kernels in terms of nearly invariant subspaces with finite defect (Theorem \ref{R44}).

Further, we extend the realization framework to TTOs on multiband spaces. By identifying the multiband space with a vector-valued model space, we obtain an analogous exact kernel representation for finite-rank perturbed TTO (Theorem \ref{R51}). As a consequence, we establish that the kernel of finite-rank perturbed TTO with rank at most $n$ is nearly backward shift invariant with defect at most \(2n+2\) (Theorem \ref{R52}). These results distinguish the present work from earlier studies of finite-rank perturbations of Toeplitz operators, which act on the  Hardy space. Our contribution concerns arbitrary finite-rank perturbations of TTOs and provides both an exact kernel realization and explicit quantitative bounds for their nearly invariant structure.

\section{Preliminaries}\label{Pre}

Throughout the paper, let $\mathbb D$ and $\mathbb T$ denote the open unit disc and the unit circle centred at the origin in the complex plane $\mathbb C$, respectively. For \(m\in\mathbb N\), the $\mathbb C^m$-valued Lebesgue space, denoted by
\(L^2(\mathbb T,\mathbb C^m)\), is defined as
\[
L^2(\mathbb T,\mathbb C^m)
=
\left\{
f=(f_1,\ldots,f_m):
f_j\in L^2(\mathbb T),\quad 1\leq j\leq m
\right\}.
\]
It is a Hilbert space with inner product
\[
\langle f,g\rangle_{L^2(\mathbb T,\mathbb C^m)}
=
\sum_{j=1}^m
\langle f_j,g_j\rangle_{L^2(\mathbb T)},
\]
where $f=(f_1, \dots , f_m)$ and $g=(g_1, \dots , g_m)$ are in $L^2(\bb T, \bb C^m)$.
Every $f$ in \(L^2(\mathbb T,\mathbb C^m)\) can be represented by its Fourier series
\[
f(z)=\sum_{k\in\mathbb Z}A_kz^k,
\]
where $A_k\in\mathbb C^m$ and
\[
\sum_{k\in\mathbb Z}\|A_k\|_{\mathbb C^m}^2<\infty.
\]

The vector-valued Hardy space is a closed subspace of \(L^2(\mathbb T,\mathbb C^m)\) consisting of functions whose negative Fourier coefficients vanish. It will serve as the ambient Hilbert space for the operator-theoretic considerations in this paper.

\begin{definition}
The $\mathbb{C}^m$-valued Hardy space, denoted by
$H^2(\mathbb{D},\mathbb C^m)$, is the Hilbert space of $\mathbb C^m$-valued analytic functions defined on $\mathbb D$ with square summable Fourier coefficients in their power series representations. Thus,
\begin{eqnarray*}
 H^2(\mathbb{D},\mathbb C^m):=
 \Bigg\{F=
 \sum_{i=0}^{\infty}A_iz^i :
 A_i \in \mathbb C^m,\ 
 \sum_{i=0}^{\infty}\|A_i\|_{\mathbb C^m}^2 < \infty
 \Bigg\}.
\end{eqnarray*}
For any
$F(z)=\sum_{i=0}^{\infty}A_i z^i$ and
$H(z)=\sum_{i=0}^{\infty}B_i z^i$ in
$H^2(\mathbb{D},\mathbb C^m)$, the inner product is given by
\begin{eqnarray*}
\langle F,H\rangle
=
\sum_{i=0}^{\infty}
\langle A_i,B_i\rangle_{\mathbb C^m}.
\end{eqnarray*}
\end{definition}
We identify \(H^2(\mathbb D,\mathbb C^m)\) with its boundary functions in
\(L^2(\mathbb T,\mathbb C^m)\).
 With this identification, let
\[
P_+:L^2(\mathbb T,\mathbb C^m)
\longrightarrow H^2(\mathbb D,\mathbb C^m)
\]
denote the orthogonal projection onto
\(H^2(\mathbb D,\mathbb C^m)\). Thus, for
\[
f(z)=\sum_{k\in\mathbb Z}A_kz^k,
\]
we have
\[
P_+f(z)=\sum_{k=0}^{\infty}A_kz^k.
\]

We next introduce matrix-valued functions, which will be used in the definition of Toeplitz and TTOs. The space of \(n\times n\) matrices with each entry in \(L^\infty(\mathbb T)\) is denoted by
\(L^\infty(\mathbb T)^{n\times n}\).

\begin{definition}
Let $G\in L^\infty(\mathbb T)^{m\times m}$. The Toeplitz operator with symbol $G$ on
$H^2(\mathbb D,\mathbb C^m)$ is the bounded operator defined by
\[
T_GF=P_+(GF),\qquad F\in H^2(\mathbb D,\mathbb C^m),
\]
where $P_+$ denotes the orthogonal projection from
$L^2(\mathbb T,\mathbb C^m)$ onto
$H^2(\mathbb D,\mathbb C^m)$.
\end{definition}

The vector-valued Hardy space also provides the natural setting for matrix-valued inner functions and their associated model spaces. Since these model spaces play a central role in the matrix-valued TTOs considered below, we recall the relevant definitions.

The space of \(n\times n\) matrices with each entry in
\(H^\infty(\mathbb D)\) is denoted by
\(H^\infty(\mathbb D)^{n\times n}\). A matrix-valued function
\(\Theta\) in \(H^\infty(\mathbb D)^{n\times n}\) is called ``inner" if the
Toeplitz operator \(T_\Theta\) is an isometry. Let \(\Theta\) be a
matrix-valued inner function, then the associated vector-valued model space is defined by
\[
\cl K_\Theta
:=
H^2(\mathbb D,\mathbb C^m)
\ominus
\Theta H^2(\mathbb D,\mathbb C^m).
\]

We will also need the shift and backward shift operators. The latter is closely connected with the notion of nearly invariant subspaces used throughout the paper.

\begin{definition}
The \textit{shift operator} on
$H^2(\mathbb{D},\mathbb{C}^m)$, denoted by $S$, is defined by
\begin{eqnarray*}
    (SF)(z)=zF(z),
    \qquad F\in H^2(\mathbb{D},\mathbb{C}^m),
\end{eqnarray*}
and its adjoint, denoted by $S^*$ and called the \textit{backward shift operator}, is given by
\begin{eqnarray*}
    (S^*F)(z)=\frac{F(z)-F(0)}{z},
    \qquad F\in H^2(\mathbb{D},\mathbb{C}^m).
\end{eqnarray*}
\end{definition}

Throughout the paper, the term ``subspace'' refers to a closed subspace. The backward shift leads naturally to the notion of nearly invariant subspaces. Unlike invariant subspaces, a nearly invariant subspace is required to be preserved by the backward shift only for those elements that vanish at the origin, possibly up to a finite-dimensional defect. This finite-dimensional defect will be essential in our analysis of finite-rank perturbations.

\begin{definition}
A subspace $\mathcal M$ of $H^2(\mathbb D,\mathbb C^m)$ is said to be
\emph{nearly $S^*$-invariant with finite defect $d$} if there exists a
$d$-dimensional subspace $\mathcal D$ (orthogonal to $\mathcal M$) such that
\begin{align}\label{FD}
F \in \cl M \text{ with } F(0)=0
\implies
S^*(F)\in\mathcal M\oplus\mathcal D.
\end{align}
A finite-dimensional subspace of the smallest dimension that satisfies
\eqref{FD} is called the \textit{defect space} of $\mathcal M$, and its
dimension is called the \textit{defect} of $\mathcal M$. If $d=0$, the we say $\cl M$ is a nearly $S^*$-invariant subspace.
\end{definition}

We now turn to TTOs. These operators are obtained by compressing multiplication operators to model spaces and constitute the principal operator class considered in this paper. We begin with the scalar-valued setting.

\begin{definition}
Let $\theta$ be an inner function and let $g\in L^\infty(\mathbb T)$.
Then the TTO with symbol $g$ on $\cl K_\theta$ is defined by
$$
A_g^\theta f
=
P_\theta(gf),
\qquad
f\in \cl K_\theta,
$$
where $P_\theta$ denotes the orthogonal projection from
$L^2(\mathbb T)$ onto $\cl K_\theta$.
\end{definition}
The scalar-valued theory extends naturally to the vector-valued setting by replacing scalar inner functions and symbols with matrix-valued ones. Since this formulation provides the framework for realizing the finite-rank perturbations considered in this paper, we recall it next.

\begin{definition}
Let $\Theta$ be an $n\times n$ matrix-valued inner function and let
$G\in L^\infty(\mathbb T)^{n\times n}$. Then the matrix-valued TTO with symbol $G$ on $\cl K_\Theta$ is defined by
$$
A_G^\Theta F
=
P_\Theta(GF),
\qquad
F\in \cl K_\Theta,
$$
where $P_\Theta$ denotes the orthogonal projection from
$L^2(\mathbb T,\mathbb C^n)$ onto $\cl K_\Theta$.
\end{definition}

Thus, the matrix-valued formulation retains the compression structure of the scalar TTO while allowing interactions between different coordinates. This additional flexibility will be useful when a TTO is coupled with the finite-dimensional coordinates arising from a perturbation.

We next introduce the finite-rank operators that will be used throughout the paper. Rank-one operators provide the basic building blocks for the perturbations under consideration.

\begin{definition}
Let $\mathcal H$ be a Hilbert space and let $p,q\in\mathcal H$. The
\emph{rank-one operator} $p\otimes q$ on $\mathcal H$ is defined by
$$
(p\otimes q)f
=
\langle f,q\rangle p,
\qquad
f\in\mathcal H.
$$
In general, a finite-rank operator of rank at most $n$ can be written in the
form
$$
\sum_{j=1}^n p_j\otimes q_j,
\qquad
p_j,q_j\in\cl H.
$$
\end{definition}

With this notation, the finite-rank perturbations considered in the paper take the form
$$
A_g^\theta-\sum_{j=1}^n p_j\otimes q_j,
\qquad
p_j,q_j\in \cl K_\theta.
$$
The vectors $p_j$ and $q_j$ allows these perturbation terms to be incorporated into a matrix-valued symbol on a suitably enlarged model space. More precisely, each rank-one term introduces a finite-dimensional coordinate. This suggests enlarging the scalar model space $\cl K_\theta$ by adjoining a finite-dimensional component. To implement this idea, we consider the following matrix-valued inner function:
$$
\Theta_n
=
\operatorname{diag}
\big(
\theta,
\underbrace{z,\ldots,z}_{n\text{ copies}}
\big).
$$
The corresponding model space is given by
\[
\cl K_{\Theta_n}
=
\cl K_\theta\oplus\mathbb C^n.
\]

This particular choice of $\Theta_n$ is convenient because
$K_z=\mathbb C$, so that each additional copy of $z$ contributes precisely one finite-dimensional coordinate. Consequently, the $n$-dimensional component $\mathbb C^n$ records the finite-rank data of the perturbation, while the first coordinate retains the original model space $\cl K_\theta$. This enlargement will be used in Section \ref{Tru} to realize the finite-rank perturbation as a matrix-valued TTO.

We conclude this section by introducing TTOs on multiband spaces.

\begin{definition}
Let $\varphi,\psi\in L^\infty(\mathbb T)$ be unimodular functions satisfying
$\varphi\cl K_\theta\perp\psi \cl K_\theta$, then the subspace
$\mathcal M:=\varphi \cl K_\theta\oplus\psi \cl K_\theta$ is called a multiband space.
\end{definition}

Every element $f\in\mathcal M$ has a unique representation
$f=\varphi f_1+\psi f_2$, where $f_1,f_2\in \cl K_\theta$. Moreover, the orthogonality of the two summands and the unimodularity of $\varphi$ and $\psi$ give
\[
\|\varphi f_1+\psi f_2\|^2
=
\|f_1\|_2^2+\|f_2\|_2^2.
\]
Hence the map
\[
U:\mathcal M\longrightarrow \cl K_\theta(\mathbb D,\mathbb C^2)
\]
defined by
$$
U(\varphi f_1+\psi f_2)
=
\begin{pmatrix}
f_1\\
f_2
\end{pmatrix}
$$
is unitary.

\begin{definition}
Suppose $g\in L^\infty(\mathbb T)$, then the TTO on the multiband space $\mathcal M$ is defined by
$$
A_g^{\mathcal M}f
=
P_{\mathcal M}(gf),
\qquad
f\in\mathcal M,
$$
where $P_{\mathcal M}$ denotes the orthogonal projection onto $\mathcal M$.
\end{definition}

Under the unitary map $U$, the operator $A_g^{\mathcal M}$ is represented by a $2\times2$ matrix-valued TTO on $\cl K_\theta(\mathbb D,\mathbb C^2)$. Thus, for a suitable matrix-valued symbol
$$
G=
\begin{pmatrix}
g_{11}&g_{12}\\
g_{21}&g_{22}
\end{pmatrix},
\qquad
g_{ij}\in L^\infty(\mathbb T),
$$
we have
$$
UA_g^{\mathcal M}U^*
=
A_G^\theta.
$$

For a perturbation of rank at most $n$ in the multiband setting, the same enlargement principle gives
$$
\Phi_n
=
\operatorname{diag}
\big(
\theta,\theta,
\underbrace{z,\ldots,z}_{n\text{ copies}}
\big),
$$
for which
\[
\cl K_{\Phi_n}
=
\cl K_\theta(\mathbb D,\mathbb C^2)\oplus\mathbb C^n.
\]
This enlarged matrix-valued model space provides the framework used in Section \ref{Mul} to analyze kernels of finite-rank perturbations of TTOs on multiband spaces.

\section{Matrix-valued Truncated Toeplitz Operators}\label{Mat}
The kernel of a matrix-valued truncated Toeplitz operator is naturally
connected with the kernel of a suitable Toeplitz operator acting on a
larger vector-valued Hardy space. This connection provides a way to
transfer structural properties of Toeplitz kernels to matrix-valued
truncated Toeplitz operators. In particular, by realizing the latter as
coordinate projection of Toeplitz kernels, we can bring the nearly
$S^*$-invariant structure of Toeplitz kernels into the truncated setting.
The following lemma provides the realization that will be used throughout
the subsequent analysis.

\begin{lemma}[Toeplitz-kernel realization] \label{R31}
Suppose $\Theta$ is an $n\times n$ matrix-valued inner function and let
$G\in  L^\infty(\mathbb T)^{n\times n}$. Define $\mathcal G:=
\begin{pmatrix}
\Theta^*&0\\
G&\Theta
\end{pmatrix}$, where $\Theta ^*$ denotes the conjugate transpose of $\Theta$.
Then
\begin{align*}
   \ker A_G^\Theta = P_n(\ker T_{\mathcal G}), 
\end{align*}
where $P_n:H^2(\mathbb D,\mathbb C^{2n})
\longrightarrow H^2(\mathbb D,\mathbb C^n)$
denotes the orthogonal projection onto the first $n$ coordinates.
\end{lemma}

\begin{proof}
Suppose that $ F\in\ker T_{\mathcal G}$, then there exist $f,h \in H^2(\mathbb D,\mathbb C^{n})$ such that $\displaystyle{F=\binom{f}{h}\in H^2(\mathbb D,\mathbb C^{2n})}$ and 
\begin{align} \label{Leq}
  P_+(\Theta^*f)=0 \text{ and } P_+(Gf + \Theta h)=0,
\end{align}
where $P_+$ is the orthogonal projection of $L^2(\bb T, \bb C^n)$ onto $H^2(\bb D, \bb C^n)$.
The first condition in (\ref{Leq}) shows that $\Theta^*f \in L^2(\bb T, \bb C^n) \ominus H^2(\bb D, \bb C^n)$. Now, for any $u \in H^2(\bb D, \bb C^n)$
\begin{align*}
 \langle f, \Theta u \rangle = \langle  \Theta^*f, u \rangle= 0.
 \end{align*}
This shows that $f \in \cl K_{\Theta}$. Further, the second condition in (\ref{Leq}) gives $P_+(Gf + \Theta h)=0$, which shows that  $P_+(Gf)=-\Theta h\in\Theta H^2(\bb D, \bb C^n)$. Also, since $P_{\Theta}= P_{\Theta} P_{+}$ on $L^2(\bb T, \bb C^n)$, therefore we have 
\begin{align*}
A_G^\Theta f
=
P_\Theta P_+(Gf)
=0.   
\end{align*}
This shows that $f=P_n(F) \in \ker A^\Theta _G$. Therefore, we obtain
\begin{align} \label{eqn31}
 P_n(\ker T_{\mathcal G})
\subseteq\ker A_G^\Theta .  
\end{align}

 Conversely, suppose $f\in\ker A_G^\Theta$. Then $P_\Theta P_+(Gf)=0$. This shows that $P_+(Gf)\in\Theta H^2(\bb D, \bb C^n)$. Thus, there exists $h\in H^2(\mathbb D,\mathbb C^n)$ such that $P_+(Gf)=-\Theta h$. Consequently,
 \begin{align} \label{eqn32}
   P_+(Gf+\Theta h)=0.  
 \end{align}
Also, since $f\in \cl K_\Theta$, therefore
\begin{align} \label{eqn33}
 P_+(\Theta^* f)=0.   
\end{align}
Thus from (\ref{eqn32}) and (\ref{eqn33}), we obtain 
 $\displaystyle{\binom{f}{h}\in\ker T_{\mathcal G}}$, and therefore $f\in P_n(\ker T_{\mathcal G})$. Consequently 
\begin{align} \label{eqn34}
    \ker A_G^\Theta \subset P_n(\ker T_{\mathcal G}). 
\end{align}
Hence, from (\ref{eqn31}) and (\ref{eqn34}) , we conclude $\ker A_G^\Theta=P_n(\ker T_{\mathcal G})$.

\end{proof}

The preceding realization allows us to exploit the structure of kernels of matrix-valued truncated Toeplitz operators in terms of certain Toeplitz kernels. The following result shows that, for a matrix-valued inner function of the from $\Theta_n
=
\operatorname{diag}
\big(
\theta,
\underbrace{z,\ldots,z}_{n\text{ copies}}
\big)$, this structure yields a near $S^*$-invariance with finite defect, and the associated defect is controlled explicitly by the number of additional coordinates.

\begin{theorem}\label{R32}
Suppose $\Theta_n
=
\operatorname{diag}
\big(
\theta,
\underbrace{z,\ldots,z}_{n\text{ copies}}
\big)$ is an $(n+1)\times(n+1)$ matrix-valued inner function, and let $G_n\in L^\infty(\mathbb T)^{(n+1)\times(n+1)}$. Then $\ker A_{G_n}^{\Theta_n}$ is a nearly $S^*$-invariant subspace with defect at most $n+1$.
\end{theorem}

\begin{proof}
Firstly, put $N=n+1$ and consider the $2N\times 2N$ matrix-valued symbol 
\begin{align*}
\mathcal G_n
=
\begin{pmatrix}
\Theta^*_n & 0\\
G_n & \Theta_n
\end{pmatrix}.    
\end{align*}
Let $T_{\mathcal G_n}$ denotes the corresponding Toeplitz operator on $H^2(\mathbb D,\mathbb C^{2N})$. Then by Lemma \ref{R31}, $$\ker A_{G_n}^{\Theta_n}= P_N\left(\ker T_{\mathcal G_n}\right),$$ where $P_N:
H^2(\mathbb D,\mathbb C^{2N})
\longrightarrow H^2(\mathbb D,\mathbb C^{N})$ is the orthogonal projection onto the first $N$ coordinates.

Define $\cl W:=\ker T_{\mathcal G_n}
\ominus\left(\ker T_{\mathcal G_n} \cap zH^2(\mathbb D,\mathbb C^{2N})
\right)$, and consider the point evaluation map $$\phi : \cl W \to \mathbb C^{2N} \text{ defined by } \phi(F)=F(0).$$ Indeed, if $F\in W$ such that $F(0)=0$, then $F\in zH^2(\mathbb D,\mathbb C^{2N})$, and therefore $F\in
\ker T_{\mathcal G_n} \cap zH^2(\mathbb D,\mathbb C^{2N})$. Thus, we must have  $F=0$. This shows that $\phi$ is injective. Consequently,
\begin{align*}
    \dim \cl W=\dim \cl W(0)=r \text{ (say) }.
\end{align*}
Suppose $f\in \ker A_{G_n}^{\Theta_n}$ with $f(0)=0$, then there exists $F \in \ker T_{\mathcal G_n}$ such that $P_NF=f$. Choose $W_1,\ldots,W_r\in \cl W$ such that $W_1(0),\ldots,W_r(0)$ 
form a basis of $\cl W(0)$. Using the decomposition 
\begin{align*}
 \ker T_{\mathcal G_n}=\cl W \oplus
\left(
\ker T_{\mathcal G_n} \cap
zH^2(\mathbb D,\mathbb C^{2N})
\right),   
\end{align*}
we can express $F$ as
\begin{align*}
 F=\sum_{j=1}^r a_jW_j+F_0,   
\end{align*}
where $a_1,\ldots,a_r\in\mathbb C$ and $F_0\in
\ker T_{\mathcal G_n}  \cap
zH^2(\mathbb D,\mathbb C^{2N})$. Also, since $f(0)=P_NF(0)=0$ and $F_0(0)=0$, we obtain $P_N\left(\sum_{j=1}^r a_jW_j(0)
\right)=0$. That is; $P_N
\left(
\sum_{j=1}^r a_jW_j
\right)
\in
zH^2(\mathbb D,\mathbb C^N)$. Therefore, it follows that
\begin{align*}
 S^*f
=
\frac{f}{z}
=
\frac{
P_N\left(\sum_{j=1}^r a_jW_j\right)
}{z}
+
P_N\left(\frac{F_0}{z}\right).   
\end{align*}
Now, since $F_0\in \ker T_{\mathcal G_n}$ with $F_0(0)=0$ and $\ker T_{\mathcal G_n} $ is nearly $S^*$-invariant, therefore we have
\begin{align*}
 \frac{F_0}{z}
=
S^*F_0
\in \ker T_{\mathcal G_n},  
\end{align*}
and hence
\begin{align*}
 P_N\left(\frac{F_0}{z}\right)
\in
P_N(\ker T_{\mathcal G_n})
=\ker A_{G_n}^{\Theta_n}.   
\end{align*}
Therefore,
\begin{align*}
  S^*f\in \ker A_{G_n}^{\Theta_n}+\mathcal D,  \text{ where } \mathcal D
=
\left(
\frac{
\operatorname{span}
\{P_NW_1,\ldots,P_NW_r\}
}{z}
\cap
H^2(\mathbb D,\mathbb C^N)
\right).
\end{align*}
The finite-dimensional space $\cl D$ is a valid defect space in the sense of the defining near invariance; if necessary, it may be replaced by its orthogonal
projection onto the orthogonal complement of $\ker A_{G_n}^{\Theta_n}$,
without increasing its dimension.

Furthermore, we estimate the dimension of $\cl D$. Consider the  $2N\times r$ matrix $V:=\begin{bmatrix}
W_1(0)&\cdots&W_r(0)
\end{bmatrix}$. Since $W_1(0),\ldots,W_r(0)$ are linearly independent, therefore $\operatorname{rank}V=r$. Now, a linear combination $\sum_{j=1}^r a_jP_NW_j$
belongs to $zH^2(\mathbb D,\mathbb C^N)$ if and only if its value at $0$ is zero. That is,
\begin{align*}
    \sum_{j=1}^r a_jP_NW_j
\in
zH^2(\mathbb D,\mathbb C^N) \iff P_NV
\begin{pmatrix}
a_1\\
\vdots\\
a_r
\end{pmatrix}
=0.
\end{align*}
This shows that $\dim\mathcal D
\leq
\dim\ker(P_NV)$. Now, by the rank-nullity theorem, $\dim\ker(P_NV)=
r-\operatorname{rank}(P_NV)$. Write $V=
\begin{pmatrix}
V_1\\
V_2
\end{pmatrix}$, where $V_1=P_NV$ consists of the first $N$ rows of $V$, while $V_2$ consists of the remaining $N$ rows. Since $\operatorname{rank}V=r$, therefore $r = \operatorname{rank}V
\leq \operatorname{rank}V_1+ \operatorname{rank}V_2$. But $V_2$ has only $N$ rows, and therefore $\operatorname{rank}V_2\leq N$. This shows that $   r
\leq
\operatorname{rank}(P_NV)+N$; that is, $ r-\operatorname{rank}(P_NV)\leq N$.
Consequently,
\begin{align*}
 \dim\mathcal D\leq N=n+1.   
\end{align*}
Hence, we have shown $\ker A_{G_n}^{\Theta_n}$ is a nearly $S^*$-invariant subspace with defect at most $n+1$.
\end{proof}

\section{Kernels of finite rank perturbations of Truncated Toeplitz Operators}\label{Tru}

Consider an inner function $\theta$ and the associated model space $\cl K_\theta :=H^2(\bb D)\ominus \theta H^2(\bb D)$. Let $P_\theta:L^2(\mathbb T)\to \cl K_\theta$ is the orthogonal projection onto $\cl K_\theta$. Then for $g\in L^\infty(\mathbb T)$,  the truncated Toeplitz operator on $\cl K_\theta$, denoted by $ A_g^\theta$,  is defined as 
\begin{align*}
    A_g^\theta f=P_\theta(gf),\qquad f\in \cl K_\theta.
\end{align*}
Let $n \in \mathbb{N}$ and  $p_j,q_j \in  \cl K_\theta$ for $i=1, \dots n$. The perturbed truncated Toeplitz operator of rank at most $n$ is given by
\begin{align*}
    \widetilde{A_n}:=A_g^\theta-\sum_{j=1}^n p_j\otimes q_j,
\end{align*}
where $(p_j\otimes q_j)f=\langle f,q_j\rangle p_j$.

Further we define the two crucial $(n+1) \times (n+1)$ matrices $\Theta_n$ and $G_n$ as follows: 
\begin{align*}
 \Theta_n:=
\begin{pmatrix}
\theta & 0 & \cdots & 0\\
0 & z & \cdots & 0\\
\vdots & \vdots & \ddots & \vdots\\
0 & 0 & \cdots & z
\end{pmatrix}_{(n+1) \times (n+1)},   
\end{align*}
where the scalar function $z$ occurs $n$ times on the diagonal,
and a matrix-valued function
\begin{align*}
G_n:=
\begin{pmatrix}
g & -p_1 & -p_2 & \cdots & -p_n\\
\overline{q_1} & -1 & 0 & \cdots & 0\\
\overline{q_2} & 0 & -1 & \cdots & 0\\
\vdots & \vdots & \vdots & \ddots & \vdots\\
\overline{q_n} & 0 & 0 & \cdots & -1
\end{pmatrix}_{(n+1) \times (n+1)}.    
\end{align*}
Since $\cl K_z=H^2(\bb D) \ominus zH^2(\bb D)=\mathbb C$, therefore
the model space associated to $\Theta_n$, denoted by $\cl K_{\Theta_n}$, is given by $$\cl K_{\Theta_n}=\cl K_\theta\oplus\mathbb C^n.$$ Then the matrix-valued truncated Toeplitz operator $A^{\Theta_n}_{G_n}: \cl K_{\Theta_n} \longrightarrow  \cl K_{\Theta_n}$ is defined by 
\begin{align*}
   A^{\Theta_n}_{G_n}(F)&=P_{\Theta_n}(G_n F), \text{ where } F \in \cl K_{\Theta_n}.
\end{align*}
That is, for an element $F=
\begin{pmatrix}
f\\ c_1\\ \vdots\\ c_n
\end{pmatrix}$ in $ \cl K_{\Theta_n}$, where $f \in \cl K_\theta$ and $c_1, \dots , c_n \in \mathbb{C}$, one has 
\begin{align*}
 A_{G_n}^{\Theta_n}(F)
=P_{\Theta_n}\begin{pmatrix}
g & -p_1 & -p_2 & \cdots & -p_n\\
\overline{q_1} & -1 & 0 & \cdots & 0\\
\overline{q_2} & 0 & -1 & \cdots & 0\\
\vdots & \vdots & \vdots & \ddots & \vdots\\
\overline{q_n} & 0 & 0 & \cdots & -1
\end{pmatrix}
\begin{pmatrix}
f\\ c_1\\c_2\\ \vdots\\ c_n\\
\end{pmatrix}=
\begin{pmatrix}
P_\theta\bigg(
gf-\displaystyle\sum_{j=1}^n c_jp_j
\bigg)
\\
P_z\left(\overline{q_1}f-c_1\right)
\\
\vdots
\\
P_z\left(\overline{q_n}f-c_n\right)
\end{pmatrix}.
\end{align*}

\vspace{0.4 cm}

We now use the enlarged model space to encode the finite-rank
perturbation. The purpose of this construction is to investigate 
the nearly $S^*$-invariant structure of kernels for these finite rank perturbations. By introducing one additional coordinate for each
rank-one term, we obtain a matrix-valued truncated Toeplitz operator whose
kernel projects onto the kernel of the original perturbed operator. The
following theorem makes this correspondence precise.

\begin{theorem} \label{R41}
 Suppose $\theta$ is an inner function, $g\in L^\infty(\mathbb T)$, and $p_j,q_j\in \cl K_\theta$ for $j=1, \dots , n$. Let $ \widetilde{A_n}$ be the perturbed truncated Toeplitz operator given by $\displaystyle{\widetilde{A_n}:=A_g^\theta-\sum_{j=1}^n p_j\otimes q_j}$. Then $$\ker  \widetilde{A_n}=P_1\bigl(\ker A_{G_n}^{\Theta_n}\bigr),$$ where $P_1$ denotes the orthogonal projection onto the first coordinate.   
\end{theorem}
\begin{proof}
Suppose $F \in  \cl K_{\Theta_n}$, then there exist $f \in \cl K_\theta$ and $c_1, \dots , c_n \in \mathbb{C}$ such that $F=
\begin{pmatrix}
f\\ c_1\\ \vdots\\ c_n
\end{pmatrix}$. Then by the definition of a matrix-valued truncated Toeplitz operator,
\[
A_{G_n}^{\Theta_n}F
=
P_{\Theta_n}(G_nF).
\]
Since the first component belongs to $L^2(\mathbb T)$ and the remaining
components are projected onto $\cl K_z=H^2(\bb D) \ominus zH^2(\bb D)=\mathbb C$, we obtain
\[
A_{G_n}^{\Theta_n}F
=
\begin{pmatrix}
P_\theta\bigg(
gf-\displaystyle\sum_{j=1}^n c_jp_j
\bigg)
\\[2mm]
P_z\left(\overline{q_1}f-c_1\right)
\\
\vdots
\\
P_z\left(\overline{q_n}f-c_n\right)
\end{pmatrix}.
\]
Now, we have $ P_z(\overline{q_j}f)=\langle f,q_j\rangle$, therefore
\begin{align*}
A_{G_n}^{\Theta_n}F
=
\begin{pmatrix}
A_g^\theta f-\displaystyle\sum_{j=1}^n c_jp_j\\[2mm]
\langle f,q_1\rangle-c_1\\
\vdots\\
\langle f,q_n\rangle-c_n
\end{pmatrix}.    
\end{align*}
Now, first suppose that $f\in\ker \widetilde{A_n}$. Then
\begin{align*}
 A_g^\theta f-
\sum_{j=1}^n\langle f,q_j\rangle p_j=0.   
\end{align*}
Choosing $c_j=\langle f,q_j\rangle$, we obtain
\begin{align*}
   A_{G_n}^{\Theta_n}
\begin{pmatrix}
f\\
\langle f,q_1\rangle\\
\vdots\\
\langle f,q_n\rangle
\end{pmatrix}
=0. 
\end{align*}
Therefore 
\begin{align*}
  \begin{pmatrix}
f\\
\langle f,q_1\rangle\\
\vdots\\
\langle f,q_n\rangle
\end{pmatrix}
\in\ker A_{G_n}^{\Theta_n}.  
\end{align*}

Conversely, suppose that
\[
\begin{pmatrix}
f\\
c_1\\
\vdots\\
c_n
\end{pmatrix}
\in\ker A_{G_n}^{\Theta_n},
\]
where $f \in \cl K_\theta$ and $c_1, \dots , c_n \in \mathbb{C}$.
Then the first component of
$A_{G_n}^{\Theta_n}F=0$ gives 
\begin{align*}
  A_g^\theta f
=
\sum_{j=1}^n
\langle f,q_j\rangle p_j,  
\end{align*}
which is precisely $ \widetilde{A_n}f=0$, and hence, $f \in \ker  \widetilde{A_n}$. This shows that $\ker  \widetilde{A_n}
=
P_1\bigl(\ker A_{G_n}^{\Theta_n}\bigr)$.

\end{proof}

We next recall a result describing the behavior of nearly
\(S^*\)-invariant subspaces under coordinate projections. 

\begin{theorem}\cite[Theorem 3.1]{ROL1} \label{PVVD}
Suppose $\cl M \subset H^2(\bb D, \bb C^m)$ is a nearly $S^*$-invariant subspace with a defect space $\cl D$.  For $1\leq i\leq m$, let $P_i:H^2(\bb D,\bb C^m)\longrightarrow H^2(\bb D,\bb C^i)$ denotes the projection onto the first $i$ coordinates. Then $P_i(\cl M)$ is a nearly invariant linear subspace (not necessarily closed) with a defect space 
 \begin{center}
    $\displaystyle
    \Bigg(
    \frac{\operatorname{span}\{P_i(W_1),\dots,P_i(W_r)\}}{z}
    \cap H^2(\bb D,\bb C^i)
    \Bigg)+P_i(\cl D)$
\end{center}
where $\{W_i\}_{i=1}^{r}$ is an orthonormal basis of  $\mathcal{M}\ominus (\mathcal{M}\cap z H^2(\mathbb{D},\mathbb{C}^m))$.

\end{theorem}

The preceding theorem gives an explicit defect space for the coordinate
projection \(P_i(\mathcal M)\). For our purposes, it is useful to have a
simpler upper bound for its dimension, which we establish next.

\begin{theorem} \label{R42}
Suppose $\cl M\subset H^2(\bb D,\bb C^m)$ is a nearly $S^*$-invariant
subspace with defect $d$. Then
$P_i(\cl M)$ is nearly $S^*$-invariant with defect at most
$d+m-i$.
\end{theorem}

\begin{proof}
Suppose $\cl D$ is the defect space for $\cl M$. Consider the point evaluation map $E:\cl M\longrightarrow\bb C^m$ defined by $E(f)=f(0)$.
Define $\cl V:= E(\cl M)\cap
\bigl(\{0\}^i\times\bb C^{m-i}\bigr)$.
Since $\cl V\subseteq
\{0\}^i\times\bb C^{m-i}$, therefore $r:=\dim\cl V\leq m-i$. Let $\{v_1,\ldots,v_r\}$ be an orthonormal basis for $\cl V$. Since $v_j\in E(\cl M)$ for each $j$, there exists
$f_j\in\cl M$ such that $f_j(0)=v_j$, for $j=1,\ldots,r$.

Now, let $g\in P_i(\cl M)$ with $g(0)=0$. Then there exists $f\in\cl M$ such that $P_i f=g$. Since $g(0)=0$, we have $P_i f(0)=0$. Consequently, $f(0) \in \cl V$. Hence there exist scalars $c_1,\ldots,c_r$ such that $$f(0)=\sum_{j=1}^r c_jv_j.$$
Define $\displaystyle{h:=f-\sum_{j=1}^r c_jf_j}$. Since $f,f_1,\ldots,f_r\in\cl M$, we have $h\in\cl M$. Moreover,
\begin{align*}
h(0)
=f(0)-\sum_{j=1}^r c_jf_j(0)
=f(0)-\sum_{j=1}^r c_jv_j
=0.    
\end{align*}
It follows that there exist $p\in\cl M$ and $q\in\cl D$ such that $S^*h=p+q$.
Since $\displaystyle{h=f-\sum_{j=1}^r c_jf_j}$, therefore 
\begin{align*}
 S^*f=p + q+\sum_{j=1}^r c_jS^*f_j   
\end{align*}
Since $P_i$ commutes with $S^*$, we obtain 
\begin{align*}
 S^*g
=S^*P_i f
=P_iS^*f
=P_i p+P_i q
+\sum_{j=1}^r c_jP_iS^*f_j.   
\end{align*}
Therefore, 
\begin{align*}
 S^*g\in P_i(\cl M)+\cl E_i,   
\end{align*}
where
\[
\cl E_i
:=
P_i(\cl D)
+
\operatorname{span}
\left\{
P_iS^*f_1,\ldots,P_iS^*f_r
\right\}.
\]
This is a finite-dimensional defect subspace in the sense of the
implication defining near invariance; replacing it by its orthogonal
projection onto the orthogonal complement of $P_i(\cl M)$ does not
increase its dimension. Thus
\[
\begin{aligned}
\dim\cl E_i
&\leq
\dim P_i(\cl D)+r\\
&\leq d+r\\
&\leq d+m-i,
\end{aligned}
\]
which shows that $P_i(\cl M)$ is
nearly $S^*$-invariant with defect at most $d+m-i$. This completes the proof.
\end{proof}

We can now combine the preceding realization with the coordinate-projection
estimate to obtain the main result for scalar truncated Toeplitz operators.
It shows that the nearly \(S^*\)-invariant structure of the kernel persists
under finite-rank perturbations, with an explicit defect bound depending
only on the rank of the perturbation.

\begin{theorem} \label{R43}
 Suppose $\theta$ is an inner function, $g\in L^\infty(\mathbb T)$, and $p_j,q_j\in \cl K_\theta$ for $j=1, \dots , n$. Let $ \widetilde{A_n}$ be the perturbed truncated Toeplitz operator defined by $\displaystyle{\widetilde{A_n}:=A_g^\theta-\sum_{j=1}^n p_j\otimes q_j}$. Then $\ker  \widetilde{A_n}$ is nearly $S^*$-invariant with finite defect of at most $2n+1$.
\end{theorem}

\begin{proof}

Suppose $A_{G_n}^{\Theta_n}$ is the matrix-valued truncated Toeplitz operator associated with $\Theta_n$ and $G_n$. Then by Theorem \ref{R32}, $\ker A_{G_n}^{\Theta_n}$
is nearly $S^*$-invariant with finite defect at most $n+1$. Thus, we can apply Theorem \ref{PVVD}
with $i=1$ to show that $\ker \widetilde{A_n}=P_1(\ker A_{G_n}^{\Theta_n})$ is nearly $S^*$-invariant with finite defect. 

 Define $\displaystyle{\mathcal W_n
:= \ker A_{G_n}^{\Theta_n}\ominus
\bigl(\ker A_{G_n}^{\Theta_n}\cap zH^2(\mathbb D,\mathbb C^{n+1})\bigr)}$ and let $W_1,\ldots,W_r$ be an orthonormal basis for $\mathcal W_n$. Let $\cl D_n$ be the defect space. Then by
Theorem~\ref{PVVD}, a defect space $\cl F_n$ for $\ker \widetilde{A_n}$ is given by
\begin{align*}
 \mathcal F_n
=
\left(
\frac{
\operatorname{span}\{P_1W_1,\ldots,P_1W_r\}
}{z}
\cap H^2(\bb D)
\right)
+
P_1(\mathcal D_n). 
\end{align*}
Further, by Theorem \ref{R42}
$$\operatorname{dim}\cl F_n \le (n+1) +(n+1)-1=2n+1.$$

This completes the proof.
\end{proof}

The preceding theorem immediately yields the corresponding result for the
unperturbed truncated Toeplitz operator by taking the perturbation rank to
be zero. In this case, the defect bound becomes one. Moreover, the defect
can be described more precisely according to whether the kernel contains
a function that does not vanish at the origin. This gives the following
corollary.
\begin{corollary}
Let $\theta$ be an inner function and $g\in L^\infty(\mathbb T)$. Then
$\ker A_g^\theta$ is nearly $S^*$-invariant with defect at most $1$.
More precisely, the following assertions hold.

\begin{enumerate}
\item[(i)] If $\ker A_g^\theta$ contains a function which does not vanish at
$0$, then $\ker A_g^\theta$ is nearly $S^*$-invariant.

\item[(ii)] If every function in $\ker A_g^\theta$ vanishes at $0$, then
$\ker A_g^\theta$ is nearly $S^*$-invariant with defect at most $1$.
\end{enumerate}
\end{corollary}

\begin{proof}
Taking $n=0$ in Theorem \ref{R43}, we have $\widetilde A_0=A_g^\theta$, and hence $\ker\widetilde A_0=\ker A_g^\theta$.
Moreover, Theorem \ref{R43} gives
\[
\dim\mathcal F_0\leq 2(0)+1=1,
\]
where $\mathcal F_0$ is a defect space for $\ker A_g^\theta$.

Suppose $\displaystyle{G:=\begin{pmatrix}
\overline{\theta}&0\\
g&\theta
\end{pmatrix}}$. Then, since \(\theta\) is inner, we have
\[
\det G=\overline{\theta}\theta=1
\quad \text{a.e. on } \mathbb T.
\]
Consequently, by Corollary 3.9 in \cite{CamPar}, the Toeplitz kernel \(\mathcal \ker T_G\) is of scalar type. Since \(\mathcal \ker T_G\neq\{0\}\), it follows from Corollary 3.11 in \cite{ROL3} that $$\dim\mathcal \ker T_G(0)=1,$$ where $\mathcal \ker T_G (0)=\{F(0):F\in\mathcal \ker T_G\}$. Moreover, by Lemma 3.9 in \cite{ROL3},
\[
\dim\mathcal W=\dim\mathcal \ker T_G(0)=1,
\]
where
\[
\mathcal W=\mathcal \ker T_G\ominus
\bigl(\mathcal \ker T_G\cap zH^2(\mathbb D,\mathbb C^2)\bigr).
\]
Thus, \(\mathcal W=\operatorname{span}\{W_1\}\) for some nonzero vector \(W_1=\begin{pmatrix} w_1  \\ w_2\end{pmatrix}\in \mathcal \ker T_G  \).

Since \(\ker A_g^\theta=P_1(\mathcal \ker T_G)\) and there exists \(f\in\ker A_g^\theta\) with \(f(0)\neq0\), we must have \(w_1(0)\neq0\). Indeed, if \(w_1(0)=0\), then every element of \(\mathcal \ker T_G\) has first coordinate vanishing at the origin, contradicting the hypothesis. Therefore,
\[
\cl D=\left(
\frac{\operatorname{span}\{w_1\}}{z}
\right)\cap H^2(\bb D)=\{0\}.
\]
 Furthermore, since \(\mathcal \ker T_G\) is a Toeplitz kernel, it is nearly \(S^*\)-invariant. Thus, by Theorem \ref{PVVD}, \(\ker A_g^\theta\) is a nearly $S^*$-invariant subspace with defect space $\cl F_0$, and $\cl F_0$ is given by
\[
\begin{aligned}
\mathcal F_0
&=
\left(
\frac{\operatorname{span}\{P_1W_1\}}{z}
\cap H^2(\bb D)
\right)+P_1(\cl D)\\
&=
\left(
\frac{\operatorname{span}\{w_1\}}{z}
\cap H^2(\bb D)
\right)
=\{0\}.
\end{aligned}
\]
Hence, \(\ker A_g^\theta\) is nearly \(S^*\)-invariant.

Further, suppose that every function in $\ker A_g^\theta$ vanishes at $0$.
Then $\ker A_g^\theta\subset zH^2(\mathbb D)$, and consequently $\ker A_g^\theta\cap zH^2(\mathbb D)=\ker A_g^\theta$. It follows that
\[
\ker A_g^\theta\ominus
\bigl(\ker A_g^\theta\cap zH^2(\mathbb D)\bigr)
=\{0\}.
\]
Thus there is no contribution from the first term in the defect space
described in Theorem \ref{PVVD}. The remaining defect space is therefore
$P_1(\mathcal D_0)$, and by Theorem \ref{R43},
\[
\dim \cl F_0=\dim P_1(\mathcal D_0)\leq1.
\]
Hence, $\ker A_g^\theta$ is nearly $S^*$-invariant with defect at most $1$.
\end{proof}

For completeness, we record the cases $n=1$ and $n=2$ separately. In
these cases, the associated matrix-valued realizations and defect bounds
take particularly transparent forms.

\subsubsection{The case $n=1$}

For $n=1$, the above construction takes the following form:
\begin{align*}
 \widetilde{A_1}=A_g^\theta-p\otimes q, \text{ where } p,q\in \cl K_\theta.
\end{align*}
Thus by Theorem \ref{R43}, $\ker \widetilde{A_1}$ is nearly $S^*$-invariant with defect at most 3.
Also, in this case
\[
\Theta_1=
\begin{pmatrix}
\theta&0\\
0&z
\end{pmatrix}, 
G_1=
\begin{pmatrix}
g&-p\\
\overline q&-1
\end{pmatrix}
\]
Therefore, for $\displaystyle{F=
\begin{pmatrix}
f\\c
\end{pmatrix}
\in \cl K_{\Theta_1}=\cl K_\theta\oplus\mathbb C}$, we have

\[
A_{G_1}^{\Theta_1}
(F)
=
\begin{pmatrix}
A_g^\theta f-cp\\
\langle f,q\rangle-c
\end{pmatrix}.
\]
Consequently,
\[
f\in\ker(\widetilde{A_1}) \iff \begin{pmatrix}
f\\
\langle f,q\rangle 
\end{pmatrix} \in
\ker A_{G_1}^{\Theta_1}.
\]
Thus
\begin{align*}
 \ker\widetilde{A_1}
=
P_1\left(\ker A_{G_1}^{\Theta_1}\right).   
\end{align*}

\subsubsection{The case $n=2$:}
For $n=2$, the above construction takes the following form:
\begin{align*}
 \widetilde{A}_2=A_g^\theta-p_1\otimes q_1-p_2\otimes q_2, \text{ where } p_1,p_2,q_1,q_2\in \cl K_\theta.  
\end{align*}
Thus by Theorem \ref{R43}, $\ker \widetilde{A_2}$ is nearly $S^*$-invariant with defect at most 5.
Also, in this case
\begin{align*}
\Theta_2=
\begin{pmatrix}
\theta&0&0\\
0&z&0\\
0&0&z
\end{pmatrix},  G_2=
\begin{pmatrix}
g&-p_1&-p_2\\
\overline{q_1}&-1&0\\
\overline{q_2}&0&-1
\end{pmatrix}.   
\end{align*}
Therefore, for $\displaystyle{F=
\begin{pmatrix}
f\\c_1\\c_2
\end{pmatrix}
\in \cl K_{\Theta_2}}=\cl K_\theta\oplus\mathbb C^2$, we have
\[
A_{G_2}^{\Theta_2}
(F)
=
\begin{pmatrix}
A_g^\theta f-c_1p_1-c_2p_2\\
\langle f,q_1\rangle-c_1\\
\langle f,q_2\rangle-c_2
\end{pmatrix}.
\]
Consequently,
\[
f\in
\ker \widetilde{A_2} \iff \begin{pmatrix}
f\\
\langle f,q_1\rangle\\
\langle f,q_2\rangle
\end{pmatrix}
\in
\ker A_{G_2}^{\Theta_2}.
\]
Thus
\begin{align*}
\ker \widetilde{A_2}
=
P_1\left(
\ker A_{G_2}^{\Theta_2}
\right).    
\end{align*}

\begin{remark}
  
Thus the scalar problem is reduced to the study of the kernel of a
matrix-valued truncated Toeplitz operator of size $(n+1)\times(n+1)$. The matrix-valued theory guarantees that
\[
\ker A_{G_n}^{\Theta_n}
\]
is nearly $S^*$-invariant with a defect of dimension at most $n+1$.
After taking the first-coordinate projection, one obtains a
finite-dimensional defect space for the scalar kernel. 

The following table summarizes the relationship between the rank of the perturbation, the size of the associated matrix-valued model, and the resulting general bound on the defect dimension:
\[
\begin{array}{c|c|c}
\text{rank of perturbation}
&
\text{matrix size}
&
\text{defect bound}
\\[2mm]
\hline
1 & 2\times2 & 3\\
2 & 3\times3 & 5\\
3 & 4\times4 & 7\\
\vdots & \vdots & \vdots\\
n & (n+1)\times(n+1) & 2n+1
\end{array}
\]
  
\end{remark}

\vspace{0.3cm}

To describe the structure of the kernels obtained above, we now use the following representation theorem for nearly \(S^*\)-invariant subspaces with finite
defect in vector-valued Hardy spaces. This was described by Loughlin in \cite{ROL1} and was also independently investigated by Chattopadhyay, Das and Pradhan in \cite{CDP1}.

\begin{theorem}\cite[Theorem 3.4]{ROL1} \label{VVD}
Suppose $\cl M \subset H^2(\mathbb{D},\mathbb{C}^m)$ is a subspace that is nearly $S^*$-invariant with a finite
defect $n$, and $\{e_j\}_{j=1}^{n}$ is an orthonormal basis of the defect space. Let $\cl W$ denotes the $r$ dimensional subspace $\mathcal{M}\ominus (\mathcal{M}\cap z H^2(\mathbb{D},\mathbb{C}^m))$. Then

\begin{enumerate}
\item In the case when there are functions in $\cl M$ that do not vanish at $0$,
    \[
    \cl M=
    \left\{
    F:\ F(z)=F_0(z)k_0(z)
    +z\sum_{j=1}^{n}k_j(z)e_j(z):
    (k_0,\ldots,k_n)\in \cl K
    \right\},
    \]
    where $F_0$ is the matrix with each column being an orthonormal element
    of  $\cl W$, 
    $k_0\in H^2(\mathbb{D},\mathbb{C}^r)$,
    $k_1,\ldots,k_m\in H^2(\bb D)$, and $ \cl K$ is an $S^*$-invariant subspace of $H^2(\mathbb{D},\mathbb{C}^{r+n})$. Furthermore, for all $F \in \cl M$, $$\displaystyle{\|F\|^2= \|k_0\|^2+\sum_{j=1}^{n}\|k_j\|^2}.$$
  
    \item In the case when all functions in $\cl M$ vanish at $0$,
    \[
    \cl M=
    \left\{
    F:\ F(z)=z\sum_{j=1}^{n}k_j(z)e_j(z):
    (k_1,\ldots,k_n)\in \cl K
    \right\},
    \]
    with the same notation as in (1), except that $\cl K$ is now a 
    $S^*$-invariant subspace of $ H^2(\mathbb{D},\mathbb{C}^n)$. Furthermore, for all $F \in \cl M$,
    \[
    \|F\|^2=\sum_{j=1}^{n}\|k_j\|^2.
    \]
\end{enumerate}

\end{theorem}

The preceding theorem provides the structural representation needed for
our kernels. Since the defect estimate established above is finite, the
representation applies directly to $\ker\widetilde{A_n}$ and yields the
following description.

\begin{theorem} \label{R44}
Suppose $\widetilde A_n$ is the finite-rank
perturbation defined in Theorem \ref{R43}, and let $\mathcal F_n$ be a defect
space of $\ker\widetilde A_n$ with $\dim\mathcal F_n=l(n)$.
Let $\{e_i\}_{i=1}^{l(n)}$ be an orthonormal basis of $\mathcal F_n$.
\begin{enumerate}
\item[(i)] If there exists a function in $\ker\widetilde{A_n}$ which does not vanish at 0, then there exists an $S^*$-invariant subspace $ \cl K_n
\subseteq H^2(\mathbb{D},\mathbb{C}^{l(n)+1})$ such that 
\begin{align*}
   \ker\widetilde{A_n}
    =
    \left\{f=
    gk+
    z\sum_{i=1}^{l(n)}k_ie_i:
    (k,k_1,\ldots,k_{l(n)})\in\cl{K}_n
    \right\}  
\end{align*}    
where $g$ is the function of unit norm in $\ker\widetilde{A_n}$, orthogonal to $\ker\widetilde{A_n} \cap zH^2(\bb D)$, satisfying $g(0)>0$. Moreover, for all $f \in \ker\widetilde{A_n}$,  
\begin{align*}
    \|f\|^2= \|k\|^2 + \sum_{i=1}^{l(n)}\|k_i\|^2.
\end{align*}
\item[(ii)]  If all functions in $\ker\widetilde{A_n}$ vanish at $0$, then there exists a $S^*$-invariant subspace $\mathcal{K}_n\subset H^2(\mathbb{D},\mathbb{C}^{l(n)})$ such that
$$\ker\widetilde{A_n}=\left\{f=z\sum_{i=1}^{l(n)}k_ie_i : (k_1,\dots, k_{l(n)})\in \mathcal{K}_n\right\}.$$
 Moreover, for all $f \in \ker\widetilde{A_n}$,
\begin{eqnarray*}
    \|f\|^2=\sum_{i=1}^{l(n)}\|k_i\|^2.
\end{eqnarray*} 
Furthermore, if $\ker \widetilde{A_n}\neq\{0\}$ and $\ell$ is the greatest non-negative integer such that $ \ker \widetilde{A_n}\subseteq z^\ell H^2(\bb D)$, then $ \widetilde{\mathcal{L}}_n
:= z^{-\ell}\ker \widetilde{A_n}$ is a nearly $S^*$-invariant subspace with finite defect, and consequently, $\widetilde{\mathcal{L}}_n$ admits a representation of the form described in (i).
\end{enumerate}

\end{theorem}

\begin{proof}
We have observed that, by Theorem \ref{R43}, $\ker  \widetilde{A_n}$ is nearly $S^*$-invariant with finite defect. Now, suppose that $\ker \widetilde{A_n}$ contains a function which does
not vanish at the origin. Then the subspace $\ker \widetilde{A_n}\ominus
(\ker \widetilde{A_n}\cap zH^2(\bb D))$ is one-dimensional. Let $g$ be its normalized basis vector. Then by the first part
of Theorem \ref{VVD}, there exists an $S^*$-invariant subspace $\mathcal{K}_n
\subseteq H^2(\mathbb{D},\mathbb{C}^{l(n)+1})$ such that
\begin{align*}
 \ker \widetilde{A_n}
=
\left\{
gk+
z\sum_{i=1}^{l(n)}k_ie_i:
(k,\ldots,k_{l(n)})\in\mathcal{K}_n
\right\}.      
\end{align*}
Moreover, for every $f \in \ker \widetilde{A_n}$,  
\begin{align*}
    \|f\|^2= \|k\|^2 + \sum_{i=1}^{l(n)}\|k_i\|^2.
\end{align*}

Further, suppose that every function in $\ker \widetilde{A_n}$ vanishes at
the origin. Then, by second part of Theorem \ref{VVD}, there exits an $S^*$-invariant
subspace such that
$$\ker \widetilde{A_n}=\left\{f=z\sum_{i=1}^{l(n)}k_ie_i : (k_1,\dots, k_{l(n)})\in \mathcal{K}_n\right\}.$$
 Further, for every $f \in \ker \widetilde{A_n}
 $,
\begin{eqnarray*}
    \|f\|^2=\sum_{i=1}^{l(n)}\|k_i\|^2.
\end{eqnarray*} 

Finally, suppose that $\ker \widetilde{A_n}\neq\{0\}$ and let $\ell$ be
the greatest non-negative integer such that $\ker \widetilde{A_n}\subseteq z^\ell H^2$.
 Since multiplication by $z^\ell$ is an isometry from $H^2(\bb D)$ onto $z^\ell H^2(\bb D)$, and $\ker \widetilde{A_n}$ is closed, therefore $\widetilde{\mathcal{L}}_n=z^{-\ell}\ker \widetilde{A_n}$ is closed, and by
maximality of $\ell$, it is not contained in $zH^2(\mathbb D)$. Now, we show that $\widetilde{\mathcal{L}}_n$ is nearly
$S^*$-invariant with finite-dimensional defect. Let $h\in\widetilde{\mathcal{L}}_n$ with $ h(0)=0$. Then $z^\ell h\in \ker \widetilde{A_n} \text{ and } z^\ell h\in z^{\ell+1}H^2(\bb D)$. Consequently, $S^*(z^\ell h)=z^\ell S^*h$. Further, let the defect space of $\ker \widetilde{A_n}$ is $\cl F_n$, then 
\[
z^\ell S^*h
\in
\ker \widetilde{A_n}+\cl F_n.
\]
Hence there exist $p\in \ker \widetilde{A_n}$ and
$q\in \cl F_n$ such that $z^\ell S^*h=p + q$. 
Both $z^\ell S^*h$ and $p$ belong to $z^\ell H^2(\bb D)$.
Therefore,
\[
q=z^\ell S^*h-p
\in
\cl F_n\cap z^\ell H^2(\bb D).
\]
It follows that
\[
S^*h
\in \widetilde{\cl L_n}+
z^{-\ell}
\bigl(\cl F_n\cap z^\ell H^2(\bb D)\bigr).
\]
The second space is finite-dimensional. Hence
$\widetilde{\mathcal{L}}_n$ is nearly $S^*$-invariant with
finite-dimensional defect. Also, since $\widetilde{\mathcal{L}}_n\not\subseteq zH^2(\bb D)$, the first part of the theorem applies to
$\widetilde{\mathcal{L}}_n$. This completes the proof.
\end{proof}

\section{Finite-rank perturbations of truncated Toeplitz
operators on multiband spaces}\label{Mul}

Suppose $\theta$ is an inner function and $\phi,\psi\in L^\infty(\mathbb T)$ are two unimodular functions satisfying $\phi \cl K_\theta\perp\psi \cl K_\theta$, then we define the multiband space as
\[
\cl M:=\phi \cl K_\theta\oplus\psi \cl K_\theta.
\]
Thus every $f\in \cl M$ has a unique representation $f=\phi f_1+\psi f_2$, where $f_1,f_2\in \cl K_\theta$. Since $\phi$ and $\psi$ are unimodular and
$\phi \cl K_\theta\perp\psi \cl K_\theta$, we have
\[
\|\phi f_1+\psi f_2\|^2
=
\|f_1\|_2^2+\|f_2\|_2^2.
\]
For $g\in L^\infty(\mathbb T)$, the truncated Toeplitz operator on $\cl M$ with symbol $g$ is defined as
\[
A_g^\cl M f=P_\cl M(gf),
\qquad f\in \cl M,
\]
where $P_\cl M$ denotes the orthogonal projection from $L^2(\mathbb T)$
onto $\cl M$.
Using the unitary map $U:\cl M\longrightarrow \cl K_\theta(\mathbb D,\mathbb C^2)$ defined by
\[
U(\phi f_1+\psi f_2)
=
\begin{pmatrix}
f_1\\
f_2
\end{pmatrix},
\]
one can show that the operator $A_g^{\cl M}$ is unitarily equivalent to the
matrix-valued truncated Toeplitz operator $A_G^\theta$ acting on
$\cl K_\theta(\mathbb D,\mathbb C^2)$, where
\[
G=
\begin{pmatrix}
g&\overline{\varphi}\,g\,\psi\\
\overline{\psi}\,g\,\varphi&g
\end{pmatrix}
\in L^\infty(\mathbb T)^{2\times2}.
\]
Further, if we consider the following $4\times4$ matrix-valued symbol
\[
\mathcal G=
\begin{pmatrix}
\theta^* I_2&0\\
G&\theta I_2
\end{pmatrix},
\]
and let $T_{\mathcal G}$ denotes the corresponding Toeplitz operator. Then the kernel of the matrix-valued truncated Toeplitz operator  $A_G^\theta$ is
obtained by projecting the kernel of $T_{\mathcal G}$ onto the
first two coordinates. That is,
\[
\ker A_G^\theta=P_2(\ker T_{\mathcal G}).
\]
We quote the following result which describes the nearly \(S^*\)-invariant structure of the kernel of a \(2\times2\) matrix-valued truncated Toeplitz operator.
\begin{theorem}\cite[Theorem 5.2]{ROL1} \label{D2}
Let $G$ be a $2\times2$ matrix-valued symbol with entries in
$L^\infty(\mathbb T)$. Then $\ker A_G^\theta$ is nearly $S^*$-invariant with defect $2$.
\end{theorem}

The multiband space $\mathcal M=\phi \cl K_\theta\oplus\psi \cl K_\theta$
is not, in general, a subspace of $H^2(\mathbb D)$. Therefore the usual backward
shift operator $S^*$ on $H^2(\mathbb D)$ does not act naturally on $\mathcal M$.
That is why we transport the vector-valued backward shift to $\mathcal M$ through the unitary operator $U$.

We define $S_{\mathcal M}^*:=U^*S^*U$. Thus, for $f=\phi f_1+\psi f_2\in\mathcal M$, $$S_{\mathcal M}^*f
=\phi S^*f_1+\psi S^*f_2.$$
Accordingly, a subspace $\mathcal N$ of the multiband space $\mathcal M=\phi \cl K_\theta\oplus\psi \cl K_\theta$ is called 
nearly $S_{\mathcal M}^*$-invariant with defect at most $d$ if
$U\mathcal N$ is nearly $S^*$-invariant in
$\cl K_\theta(\mathbb D,\mathbb C^2)$ with defect at most $d$.

\begin{remark}By Theorem~\ref{D2} and the unitary equivalence induced by $U$,
$U(\ker A_g^{\cl M})$ is nearly $S^*$-invariant with defect at most $2$.
Equivalently, $\ker A_g^{\cl M}$ is nearly
$S_{\cl M}^*$-invariant with defect at most $2$.  
\end{remark}

For $u_j,v_j$ in $\cl M$, consider the finite rank operator $\displaystyle{\sum_{j=1}^n u_j\otimes v_j}$, where $(u_j\otimes v_j)f
=\langle f,v_j\rangle_{\cl M} u_j$. Then the perturbed truncated Toeplitz operator on $\cl M$ is given by 
\begin{align} \label{EqA}
 \widetilde A_n^{\cl M}
:=
A_g^\cl M-\sum_{j=1}^n u_j\otimes v_j.   
\end{align}

Let $p_j=Uu_j$, and $q_j=Uv_j$. Since $U$ is unitary, we have $U(u_j\otimes v_j)U^*
=
p_j\otimes q_j$. Thus
\[
U\widetilde A_nU^*
=
A_G^\theta-\sum_{j=1}^n p_j\otimes q_j.
\]
Therefore, if we define
\begin{align} \label{EqT}
\widetilde{T_n}
:=
A_G^\theta-\sum_{j=1}^n p_j\otimes q_j,    
\end{align}
then, we have
\[
U(\ker\widetilde A_n^\cl M)=\ker \widetilde{T_n}.
\]

\subsection{Matrix realization of the finite rank perturbation:}
Consider the finite rank perturbed truncated Toeplitz operator $\widetilde{T_n}$ on $\cl K_\theta(\mathbb D,\mathbb C^2)$. Set $\displaystyle{\Phi_n
=
\operatorname{diag}
\big(
\theta,\theta,
\underbrace{z,\ldots,z}_{n\text{ copies}}
\big)}$, then we have the following decomposition 
\[
\cl K_{\Phi_n}
=
\cl K_\theta(\mathbb D,\mathbb C^2)\oplus\mathbb C^n.
\]
For $\displaystyle{G=
\begin{pmatrix}
g_{11}&g_{12}\\
g_{21}&g_{22}
\end{pmatrix}}$, with each $g_{ij}$ in $L^\infty(\mathbb T)$ and $\displaystyle{p_j=
\begin{pmatrix}
p_{j1}\\
p_{j2}
\end{pmatrix}}$, $\displaystyle{q_j=
\begin{pmatrix}
q_{j1}\\
q_{j2}
\end{pmatrix}}$ in $\cl K_\theta(\mathbb D,\mathbb C^2)$, we define the $(n+2)\times(n+2)$ matrix-valued symbol
\[
\mathcal A_n=
\begin{pmatrix}
G&-p_1&-p_2&\cdots&-p_n\\
q_1^*&-1&0&\cdots&0\\
q_2^*&0&-1&\cdots&0\\
\vdots&\vdots&\vdots&\ddots&\vdots\\
q_n^*&0&0&\cdots&-1
\end{pmatrix}, \text{ where } \displaystyle{q_j^*
=
\begin{pmatrix}
\overline{q_{j1}}&
\overline{q_{j2}}
\end{pmatrix}}.
\]

The above construction now connects the kernel of $\widetilde{T_n}$ with
the kernel of the associated matrix-valued truncated Toeplitz operator.
More precisely, the following result shows that
$\ker\widetilde{T_n}$ is recovered as a coordinate projection of
$\ker A_{\mathcal A_n}^{\Phi_n}$.

\begin{theorem} \label{R51}
Suppose $\displaystyle{\widetilde{T_n}=A_G^\theta-\sum_{j=1}^n p_j\otimes q_j}$ is the finite-rank perturbation of the matrix-valued truncated Toeplitz operator $A_G^\theta$ on
$\cl K_\theta(\mathbb D,\mathbb C^2)$ as defined in (\ref{EqT}). Then $$\ker \widetilde{T_n}
=
P_2\left(\ker A_{\mathcal A_n}^{\Phi_n}\right),$$ where $P_2$ is the projection onto the first two coordinates.

\end{theorem}
\begin{proof}
First we observe that, since $\displaystyle{\Phi_n
=
\operatorname{diag}
\big(
\theta,\theta,
\underbrace{z,\ldots,z}_{n\text{ copies}}
\big)}$ and $\cl K_z=H^2(\bb D)\ominus zH^2(\bb D)=\mathbb C$, therefore we have the following orthogonal decomposition
$$\cl K_{\Phi_n}=\cl K_\theta(\mathbb D,\mathbb C^2)\oplus\mathbb C^n.$$ Consequently, every element \(F\in \cl K_{\Phi_n}\) can be uniquely expressed as $\displaystyle{F=
\begin{pmatrix}
f\\
c_1\\
\vdots\\
c_n
\end{pmatrix}}$, where $f\in \cl K_\theta(\mathbb D,\mathbb C^2)$ and $c_1,\ldots,c_n\in\mathbb C$. By the definition of \(\mathcal A_n\), we obtain
$$
\mathcal A_nF
=
\begin{pmatrix}
Gf-\displaystyle\sum_{j=1}^n c_jp_j\\[2mm]
q_1^*f-c_1\\
\vdots\\
q_n^*f-c_n
\end{pmatrix}.
$$
We now compress this vector onto \(\cl K_{\Phi_n}\). For the first component, the compression is precisely the compression of multiplication by \(G\) onto \(\cl K_\theta(\mathbb D,\mathbb C^2)\), namely \(A_G^\theta\). 

Further, since
$
\displaystyle{q_j^*f=\begin{pmatrix}
\overline{q_{j1}}&
\overline{q_{j2}}
\end{pmatrix}
\begin{pmatrix}
f_1\\
f_2
\end{pmatrix}
=\overline{q_{j1}}f_1 + \overline{q_{j2}}f_2},
 $
therefore, $P_z(q_j^*f)=P_z(\overline{q_{j1}}f_1)+P_z(\overline{q_{j2}}f_2)=\langle f,q_j\rangle$. Now, each of the remaining model spaces is \(\cl K_z=\mathbb C\), the corresponding scalar components already belong to \(\cl K_z\) and hence are unchanged by the compression. Thus, we obtain
$$
A_{\mathcal A_n}^{\Phi_n}F
=
\begin{pmatrix}
A_G^\theta f-\displaystyle\sum_{j=1}^n c_jp_j\\[2mm]
\langle f,q_1\rangle-c_1\\
\vdots\\
\langle f,q_n\rangle-c_n
\end{pmatrix}.
$$
It follows that
$$
F\in\ker A_{\mathcal A_n}^{\Phi_n} \iff A_G^\theta f-\sum_{j=1}^n c_jp_j=0 \text{ and } c_j=\langle f,q_j\rangle,
$$
for every $j=1,\ldots,n$.

 Now, substituting \(c_j=\langle f,q_j\rangle\), we get $\displaystyle{A_G^\theta f
-
\sum_{j=1}^n
\langle f,q_j\rangle p_j
=0}$. That is,
$$
\left(
A_G^\theta-\sum_{j=1}^n p_j\otimes q_j
\right)f=0.
$$
This shows that $\widetilde{T_n}f=0$. Therefore, for every \(f\in\ker \widetilde{T_n}\), the vector
$$
F
=
\begin{pmatrix}
f\\
\langle f,q_1\rangle\\
\vdots\\
\langle f,q_n\rangle
\end{pmatrix}
$$
belongs to \(\ker A_{\mathcal A_n}^{\Phi_n}\).

 Conversely, suppose that $\displaystyle{F=
\begin{pmatrix}
f\\
c_1\\
\vdots\\
c_n
\end{pmatrix}
\in\ker A_{\mathcal A_n}^{\Phi_n}}$. Then the lower \(n\) components imply $c_j=\langle f,q_j\rangle$ for every $j=1,\ldots,n$. Also, the first component implies $\widetilde{T_n}f=0$. Thus
$$
f\in\ker \widetilde{T_n}.
$$
Hence the first component of every vector in $\ker A_{\mathcal A_n}^{\Phi_n}$ belongs to $\ker \widetilde{ T_n}$, and every vector in $\ker \widetilde{ T_n}$  admits the above lift to $\ker A_{\mathcal A_n}^{\Phi_n}$. Thus, we obtain
$$
\ker \widetilde{T_n}
=
P_2\left(\ker A_{\mathcal A_n}^{\Phi_n}\right).
$$
This completes the proof.
\end{proof}

\subsection{Toeplitz-kernel realization:}
Set $N=n+2$, then $\mathcal A_n$ is an $N\times N$ matrix. Now, consider a  $2N\times2N$ matrix-valued symbol defined by
\[
\mathfrak G_n
:=
\begin{pmatrix}
\Phi_n^*&0\\
\mathcal A_n&\Phi_n
\end{pmatrix}.
\]
Let $T_{\mathfrak G_n}$ denotes the corresponding Toeplitz operator.
Then the standard matrix-valued truncated Toeplitz realization gives $\ker A_{\mathcal A_n}^{\Phi_n}
=
P_N(\ker T_{\mathfrak G_n})$,  where $P_N$ is the projection onto the first $N$ coordinates. Combining this identity with Theorem \ref{R51}, we get
\[
\ker \widetilde{T_n}
=P_2(\ker A_{\mathcal A_n}^{\Phi_n})=
P_2(P_N\big(\ker T_{\mathfrak G_n})\big)=P_2\big(\ker T_{\mathfrak G_n})
.
\]
This representation is the key tool in the analysis of the finite-rank
perturbation.

\begin{theorem}\label{R52}
Suppose $\displaystyle{\widetilde A_n^{\cl M}=
A_g^{\cl M}-\sum_{j=1}^n u_j\otimes v_j}$
is the perturbed truncated Toeplitz operator on the multiband space $\cl M$ as defined in (\ref{EqA}). Then $\ker\widetilde A_n^{\cl M}$ is nearly $S_{\cl M}^*$-invariant with finite defect at most $2n+2$.

\end{theorem}
\begin{proof}
Consider the unitary map $U:\cl M\longrightarrow \cl K_\theta(\mathbb D,\mathbb C^2)$ given by $$\displaystyle{U(\phi f_1+\psi f_2)
=
\begin{pmatrix}
f_1\\
f_2
\end{pmatrix}}.$$
For $1\leq j\leq n$, define $p_j:=Uu_j$, and $q_j:=Uv_j$.
Since $U$ is unitary, rank-one operators are preserved under unitary
conjugation. More precisely,
$$
U(u_j\otimes v_j)U^*
=
(Uu_j)\otimes(Uv_j)
=
p_j\otimes q_j.
$$
Therefore
$$
\begin{aligned}
U\widetilde A_nU^*
=
U\left(A_g^\cl M-\sum_{j=1}^n u_j\otimes v_j\right)U^*
=
A_G^\theta-\sum_{j=1}^n p_j\otimes q_j.
\end{aligned}
$$
Define $\displaystyle{\widetilde{T_n}:=A_G^\theta-\sum_{j=1}^n p_j\otimes q_j}$. Thus $U\widetilde A_nU^*=\widetilde{T_n}$. Since $U$ is unitary, it follows immediately that $U(\ker\widetilde A^{\cl M}_n)=\ker \widetilde{T_n}$.

Now we apply Theorem \ref{R51}. Then,
$$
\ker \widetilde{T_n}
=
P_2\left(\ker A_{\mathcal A_n}^{\Phi_n}\right),
$$
where $P_2$ denotes the projection onto the first two coordinates. Put $N=n+2$, then the associated Toeplitz-kernel realization is given by the
$2N\times2N$ matrix-valued symbol
$$
\mathfrak G_n
=
\begin{pmatrix}
\Phi_n^*&0\\
\mathcal A_n&\Phi_n
\end{pmatrix}.
$$
By the Toeplitz-kernel realization, we have $\ker A_{\mathcal A_n}^{\Phi_n}= P_N(\ker T_{\mathfrak G_n}$), and consequently
$$
\ker \widetilde{T_n}=P_2(P_N(\ker T_{\mathfrak G_n}))=P_2(\ker T_{\mathfrak G_n}).
$$
Further, since
$\ker T_{\mathfrak G_n}$ is a Toeplitz kernel, it is nearly $S^*$-invariant with defect zero. Applying Theorem \ref{R42} to $\ker T_{\mathfrak G_n}$, with $m=2N, i=2, \text{ and }d=0$, we obtain that $P_2(\ker T_{\mathfrak G_n})$ is nearly $S^*$-invariant with defect at most $2N-2$. Therefore, 
$$
\operatorname{defect}(\ker \widetilde{T_n})
=
\operatorname{defect}\bigl(P_2(\ker T_{\mathfrak G_n})\bigr)\le 2N-2=2n+2.
$$
Finally, since $U\bigl(\ker\widetilde A_n^{\mathcal M}\bigr)
=\ker\widetilde T_n$, it follows that $U\bigl(\ker\widetilde A_n^{\mathcal M}\bigr)$ is nearly $S^*$-invariant with defect at most $2n+2$. Therefore, by using the transported backward shift, $\ker\widetilde A_n^{\mathcal M}$ is nearly
$S_{\mathcal M}^*$-invariant with defect at most $2n+2$.
This completes the proof.

\end{proof}

The preceding theorem gives a general defect bound for the kernel of a finite rank perturbation on a multiband space. To make the result more
concrete, we record the cases $n=0$, $n=1$, and $n=2$ separately.

\subsubsection{The case $n=0$:}

 For $n=0$, Theorem \ref{R52} gives $$\operatorname{defect}(\ker\widetilde A^{\cl M}_0)=\operatorname{defect}(\ker A_g^\cl M)\leq2.$$ It shows that $\ker A_g^\cl M$ is nearly $S^*$-invariant with defect at most $2$. This recovers Theorem \ref{D2}.

\subsubsection{The case $n=1$}

For $n=1$, Theorem \ref{R52} gives 
\begin{align*}
   \operatorname{defect}(\ker\widetilde A^{\cl M}_1)=\operatorname{defect}
\left(
\ker(A_g^\cl M-u\otimes v)
\right)
\leq2(1)+2=4. 
\end{align*}
In this case,
\[
\widetilde{T_1}=A_G^{\Phi_{1}}-p\otimes q, \text{ where } \Phi_1=\operatorname{diag}(\theta,\theta,z)
\]
and
\[
\mathcal A_1=
\begin{pmatrix}
g_{11}&g_{12}&-p_1\\
g_{21}&g_{22}&-p_2\\
q^*_1&q^*_2&-1
\end{pmatrix}.
\]
The associated matrix $\mathfrak G_1$ for Toeplitz-kernel realization is of order $6\times6$.

\subsubsection{The case $n=2$:}
For $n=2$, Theorem \ref{R52} gives
\[
\operatorname{defect}(\ker\widetilde A^{\cl M}_2)=\operatorname{defect}
\left(
\ker
(A_g^\cl M-u_1\otimes v_1-u_2\otimes v_2)
\right)
\leq2(2)+2=6.
\]
In this case 
\begin{align*}
\widetilde{T_2}=A_G^{\Phi_{2}}-p_1\otimes q_1 - p_2 \otimes q_2   \text{ where } \Phi_2
=
\operatorname{diag}(\theta,\theta,z,z)
\end{align*}
and 
\[
\mathcal A_2=
\begin{pmatrix}
g_{11}&g_{12}&-p_{11}&-p_{21}\\
g_{21}&g_{22}&-p_{12}&-p_{22}\\
 q^*_{11}&q^*_{12}&-1&0\\
 q^*_{21}&q^*_{22}&0&-1
\end{pmatrix}.
\]

The associated matrix $\mathfrak G_2$ for Toeplitz-kernel realization is of order $8\times8$.

\begin{remark}
 In Theorem \ref{R52}, we have established the following chain of representations:

 The multiband space $\cl M=\phi \cl K_\theta\oplus\psi \cl K_\theta$ is unitarily equivalent to $\cl K_\theta(\mathbb D,\mathbb C^2)$. For the perturbation of rank at most $n$, $$\displaystyle{
A_G^\theta-\sum_{j=1}^np_j\otimes q_j},$$
we introduce $n$ auxiliary coordinates and obtain $\cl K_{\Phi_n}
=
\cl K_\theta(\mathbb D,\mathbb C^2)\oplus\mathbb C^n$, where $\Phi_n
=
\operatorname{diag}
(\theta,\theta,z,\ldots,z)$. 
The corresponding Toeplitz-kernel realization takes place in $H^2(\mathbb D,\mathbb C^{2n+4})$, and projection onto the first two coordinates gives
\[
\ker
\left(
A_G^\theta-\sum_{j=1}^np_j\otimes q_j
\right).
\]
The coordinate-projection argument consequently yields
\[
\operatorname{defect}
\left(
\ker
\left(
A_G^\theta-\sum_{j=1}^np_j\otimes q_j
\right)
\right)
\leq2n+2.
\]

Equivalently, on the original multiband space $\cl M$, we get
\[
\operatorname{defect}
\left(
\ker
\left(
A_g^\cl M-\sum_{j=1}^nu_j\otimes v_j
\right)
\right)
\leq2n+2.
\]
The following table summarizes the relationship between the rank of the perturbation, the orders of the associated matrix-valued symbols $\mathcal A_n$ and $\mathfrak G_n$, and the resulting bound on the defect dimension:
\[
\begin{array}{c|c|c|c}
\text{rank of perturbation}
&
\text{order of }\mathcal A_n
&
\text{Order of } \mathfrak G_n
&
\text{defect bound}
\\
\hline
0&2\times 2&4\times4&2\\
1&3\times3&6\times6&4\\
2&4\times4&8\times8&6\\
3&5\times5&10\times10&8\\
\vdots&\vdots&\vdots&\vdots\\
n&(n+2)\times (n+2)&2(n+2)\times2(n+2)&2n+2.
\end{array}
\]
\end{remark}

\vspace{8mm}

\subsection*{Acknowledgements}

Both authors are supported by the Core Research Grant (CRG), File No.: CRG/2023/004826, by the Science and Engineering Research Board(SERB), Department of Science \& Technology (DST), and the FIST Grant No.: SR/FST/MS-II/2024/183 (C) by the Department of Science \& Technology (DST), Government of India.

\subsection*{Declaration on the use of OpenAI}
Authors acknowledge the use of OpenAI's ChatGPT for language improvement and grammatical refinement of the manuscript. The tool was used solely to enhance the clarity, readability, and presentation of the text. The authors take full responsibility for all mathematical content, results, proofs, and arguments presented in the manuscript.

\subsection*{Statements and Declarations} 
\noindent\mbox{}\par\vspace{0.2cm}

\noindent {\bf Competing Interests}  The authors have no competing interests to declare that are relevant to the content of this article.  

\vspace{.2 cm}

\noindent {\bf Data availability statements} Data sharing is not applicable to this article as no new data were created or analyzed in this study.

\bibliography{refs.bib}

\end{document}